\documentclass[a4paper, 11pt] {article}
\usepackage{amsfonts, amsmath, amssymb, amsthm,url}
\usepackage{graphicx}
\usepackage[english] {babel}
\usepackage{enumitem}

\usepackage[ps, all,cmtip, dvips]{xy}

\usepackage{geometry}
\newtheorem{thm}{Theorem}[section]
\newtheorem{cor}[thm]{Corollary}
\newtheorem{lem}[thm]{Lemma}
\newtheorem{defi}[thm]{Definition}
\newtheorem{prop}[thm]{Proposition}

\theoremstyle{definition}
\newtheorem{rem}[thm]{Remark}
\newtheorem{ex}[thm]{Example}

\newcommand{\Cob}{\mathcal{C}}
\newcommand{\R}{\mathbb{R}}

\newcommand{\Z}{\mathbb{Z}}

\newcommand{\eps}{\varepsilon}
\DeclareMathOperator{\Th}{Th}

\DeclareMathOperator{\indre}{int}

\DeclareMathOperator{\Ob}{Ob}
\DeclareMathOperator{\Mor}{Mor}

\title{A geometric interpretation of the homotopy groups of the cobordism category}
\author{Marcel B\"{o}kstedt and Anne Marie Svane}
\date{}
\begin{document} 
\maketitle
\begin{abstract}
The classifying space of the embedded cobordism category has been identified  in \cite{GMTW} as the infinite loop space of a certain Thom spectrum. This identifies the set of path components with the classical cobordism group. In this paper, we give a geo\-me\-tric interpretation of the higher homotopy groups as certain cobordism groups where all manifolds are now equipped with a set of orthonormal sections in the tangent bundle. We also give a description of the fundamental group as a free group with a set of geometrically intuitive relations.
\end{abstract}
\section{Introduction}
Consider the embedded cobordism category $\mathcal{C}_d $ introduced in \cite{GMTW}. The objects are smooth closed $(d-1)$-manifolds embedded in $(0,1)^{n+d-1}$. 
A morphism from $M_0$ to $M_1$ is a smooth compact $d$-dimensional manifold $W$ embedded in $ (0,1)^{n +d-1}\times [0,1]$ which is cylindrical near the boundary $\partial W$ where
\begin{equation*}
\partial W = W\cap ((0,1)^{n +d-1}\times \{0,1\})=  M_0 \times \{0\}\cup  M_1\times \{1\}. 
\end{equation*}
There is also a version of $\mathcal{C}_d $ where all manifolds have orientations. 


The main result about the cobordism category is the identification of its classifying space $B\mathcal{C}_d$ proved by Galatius, Tillmann, Madsen, and Weiss in \cite{GMTW}.
In Section~\ref{embcob}, a more precise definition of the cobordism category and statement of their theorem is given.

It is immediate from this theorem that $\pi_0(B\mathcal{C}_d)$ is the usual Thom cobordism group $\Omega_{d-1}$ of $(d-1)$-dimensional manifolds. 
The goal of this paper is to give a geometric interpretation of the higher homotopy groups $\pi_r(B\mathcal{C}_d)$, $r>0$.
 
Let $M_0$ and $M_1$ be closed $(d-1)$-manifolds with $r$ pointwise linearly independent sections given in $TM_0\oplus \R$ and $TM_1 \oplus \R$, respectively. A vector field cobordism from $M_0$ to $M_1$ is a cobordism with $r$ independent sections in $TW$ extending the ones given on the boundary where $TW_{\mid M_i}$ is identified with $TM_i\oplus \R$ using the inward normal for  $i=0$ and the outward normal for $i=1$. The purpose of Section \ref{reppi1} and \ref{interprinv} below is to show:
\begin{thm} \label{vfcob.}
If $d$ is odd or $r<\frac{d}{2}$, vector field cobordism is an equivalence relation and
$\pi_{r}(B\mathcal{C}_{d-r})$ is isomorphic to the group of equivalence classes. Otherwise, the latter is true for the equivalence relation generated by vector field cobordism. 

For $d$ even or $r<\frac{d}{2}$, every equivalence class in $\pi_r(B\mathcal{C}_{d-r})$ is represented by a closed $(d-1)$-manifold $M$ with $r-1$ independent sections in $TM$ together with the normal section. 

This holds in both the oriented and unoriented situation.
\end{thm}
In an upcoming paper \cite{ABJ}, we are going to study obstructions to independent tangent vector fields on manifolds and Theorem \ref{vfcob.} will play a role in the identification of the top obstruction. 

In the last two sections of the paper, we obtain a description of $\pi_1(B\mathcal{C}_d)$ in terms of generators and relations:
\begin{thm}
$\pi_1(B\mathcal{C}_d)$ is an abelian group generated by the diffeomorphism classes $[W]$ of closed $d$-manifolds $W$. The only relations are as follows: If $W_1$ and $W_2$ are cobordisms from $\emptyset$ to $M$ and $W_3$ and $W_4$ are cobordisms from $M$ to $\emptyset$, then
\begin{equation*}
[W_1 \cup_M W_3] + [W_2 \cup_M W_4] = [W_1  \cup_M W_4] + [W_2 \cup_M W_3].
\end{equation*} 
Under the isomorphism of Theorem \ref{vfcob.}, $[W]$ corresponds to the equivalence class of $W$ with the single section $\eps$.
\end{thm}
Here $\cup_M  $ denotes the composition of morphisms given by glueing along a common boundary.
There is also a version of the theorem for manifolds with tangential structures under certain conditions. 

The authors wish to thank Johan Dupont for pointing our attention to the connection between cobordism and vector fields.

\section{The cobordism category and related spectra}\label{embcob}
We first recall the embedded cobordism category, following the definition in~\cite{monoids}. 
Let $G(d,n)$ denote the Grassmannian consisting of all $d$-dimensional subspaces of $\R^{n+d}$. 
The splitting $\R^{n+d+l}\cong \R^{n+d}\oplus \R^l$ induces an inclusion $j:G(d,n) \to G(d+l,n)$.
\begin{defi}
Let $\theta: X \to BO(d+l)$ be a fibration.
As a set, $\Psi_{\theta_{d}}(\R^{n+d})$ consists of all pairs $(M,\bar{\xi})$ where $M \subseteq \R^{n+d}$ is an embedded $d$-dimensional manifold without boundary such that $M$ is a closed subset of $\R^{n+d}$ and $\bar{\xi}$ is a
 lift under $\theta $ of the classifying map $\xi : M \to G(d,n) \xrightarrow{j} G(d+l,n)$. 
A suitable topology on $\Psi_{\theta_{d}}(\R^{n+d})$ is given in \cite{monoids}.

Let $\psi_{\theta_{d}}(n+d,k)$ denote the subspace consisting of those $M$ that are contained in $ (-1,1)^{n+d-k}\times \R^{k}$.
\end{defi}

\begin{defi}
Let $\theta : X \to BO(d)$ be a fibration.
The cobordism category $\mathcal{C}_{d,n+d}^\theta$ is a topological category with object space $\psi_{\theta_{d-1}}(n+d-1,0)$. The space of morphisms  is the disjoint union of the identity morphisms and a subspace of $\psi_{\theta_{d}}(n+d,1) \times (0,\infty)$. A pair $(W,a)$ is a morphism from $M_0$ to $M_1$ if  $W \in \psi_{\theta_{d}}(n+d,1) $ is such that for some $\epsilon > 0$, 
\begin{align*}
W \cap ( \R^{n+d-1}\times (-\infty,\epsilon)) &=   M_0 \times (-\infty, \epsilon)\\
W \cap ( \R^{n+d-1} \times (a-\epsilon, \infty))&=   M_1 \times (a- \epsilon, \infty)
\end{align*}
such that the $\theta$-structures agrees.
Composition of the morphisms $(W,a)$ and $(W',a')$ is given by $(W\circ W',a+a')$ where $W\circ W'$ is the union of $W \cap (\R^{n+d-1} \times (-\infty,a])$ and $W' \cap (\R^{n+d-1} \times [0,\infty)) + ae_{n+d}$.
\end{defi}
We leave the $a$ out of the notation for the morphisms when it plays no significant role. 
The splitting $\R^{1+n+d} = \R \oplus \R^{n+d}$ defines an inclusion $i:G(d,n) \to G(d,1+n)$ and hence an inclusion of categories $\mathcal{C}_{d,n+d}^\theta\to \mathcal{C}_{d,1+n+d}^\theta$.
We usually let $n$ tend to infinity and denote the resulting category by $\mathcal{C}_d^\theta$ with objects $\Ob(\Cob_d^\theta)$ and morphisms $\Mor(\Cob_d^\theta)$. The subspace consisting of the morphisms from $M_0$ to $M_1$ is denoted by $ \mathcal{C}_{d}^\theta(M_0,M_1)$.

Let $N_k(\mathcal{C}_d^\theta)$ be the $k$th nerve of the category. 
Then the classifying space $B\mathcal{C}_d^\theta$  is the topological space
\begin{equation*}
\bigsqcup N_k(\mathcal{C}_d^\theta) \times \Delta^k / \sim .
\end{equation*}
Here $\bigsqcup $ is disjoint union, $\Delta^k$ is the standard $k$-simplex, and the equivalence relation $\sim $ is given by the face and degeneracy operators, see e.g.\ \cite{may} for the precise relations.

The Galatius--Tillmann--Madsen--Weiss theorem is now the following theorem, proved in \cite{GMTW} and in the above set-up, in \cite{monoids}:
\begin{thm}\label{RW}
There is a weak homotopy equivalence
\begin{equation*}
\alpha_{d,\theta} : B\mathcal{C}_d^\theta \to \Omega^{\infty +d-1} \theta^*MTO(d).
\end{equation*}
\end{thm}
Here $\theta^*MTO(d)$ is the spectrum defined as follows: Let $U_{d,n} \to G(d,n)$ be the  universal bundle with complement $U_{d,n}^\perp$. Then $MTO(d)$ is the spectrum with $n$th space the  Thom space 
$\Th(U_{d,n}^\perp)$.
 The spectrum maps are induced by the inclusion $i$: 
\begin{equation*} 
\Sigma \Th(U_{d,n}^\perp) = \Th(i^*U_{d,1+n}^\perp) \to \Th(U_{d,1+n}^\perp).
\end{equation*}

Given a fibration $\theta: X \to BO(d)$, the spectrum $\theta^*MTO(d)$ is defined similarly with $n$th space
$\Th(\theta^*U_{d,n}^\perp \to \theta^{-1}(G(d,n)))$. Two important special cases are $X=BO(d)$ and $X = BSO(d)$. Since most contructions below work the same way in both cases, we shall write $MT(d)$ for the spectrum, $B(d)$ and $G(d,n)$ for the corresponding classifying spaces, and $\Cob_d$ for the cobordism category whenever there is no essential difference. 

The inclusion $G(d,n) \to G(d+1,n)$ is $d$-connected for $n$ large, so $\varinjlim_d MT(d)$ is the Thom cobordism spectrum. Thus the Pontryagin--Thom theorem identifies the lower homotopy groups $\pi_k(MT(d))$ as the oriented or unoriented cobordism group $\Omega_k$, respectively, when $k<d$. See e.g.\ \cite{stong} for more on classical cobordism theory.

In the proof of Theorem \ref{vfcob.}, we shall consider the fibration 
\begin{equation*}
V_{d,r} \to V_r(U_{d}) \xrightarrow{i_r} B(d). 
\end{equation*}
Here $V_{d,r}$ is the Stiefel manifold consisting of ordered $r$-tuples of orthonormal vectors in $\R^{d}$. For any vector bundle $E\to X$, $V_r(E)\to X$ will denote the fiber bundle with fiber over $x\in X$ the set of ordered $r$-tuples of orthonormal vectors in $E_x$. 

The corresponding cobordism category will be denoted $\mathcal{C}_d^r$. The objects are embedded compact $(d-1)$-dimensional manifolds $M$ equipped with $r$ orthonormal sections in $TM \oplus \R$. The morphisms are embedded cobordisms with $r$ orthonormal tangent vector fields extending the ones given on the boundary. 

Reading through the definition of the map $\alpha_{d,\theta}$ in Theorem \ref{RW}, one sees that there is a commutative diagram 
\begin{equation}\label{forget}
\vcenter{\xymatrix{{B\mathcal{C}_d^r}\ar[r]^-{BF}\ar[d]^{\alpha_{d,r}}&{B\mathcal{C}_d}\ar[d]^{\alpha_{d}}\\
{\Omega^{\infty +d-1} i_r^*MT(d)}\ar[r]^-{i_r}&{\Omega^{\infty +d -1} MT(d).}
}}
\end{equation}
Here $BF$ is the map induced by the functor $F$ that forgets the tangential structure. 

There is an inclusion $B(d-r)\to V_r(U_d)$ taking $P\subseteq \R^{n+d-r}$ to $P\oplus \R^r \subseteq \R^{n+d-r}\oplus \R^r$ with the $r$ standard basis vectors in $\R^r$ as $r$-frame. This is a homotopy equivalence, as it is the fiber inclusion for the fibration
\begin{align*}
B(d-r)\to V_r(U_d) \to V_{\infty ,r}
\end{align*}
where $V_{\infty ,r} = \varinjlim_n V_{n+d,r}$ is contractible. Thus the inclusion $MT(d-r)\to i_r^*MT(d)$ is a homotopy equivalence. This yields the isomorphisms
\begin{equation*}
\pi_{r+1}(B\Cob_{d-r}) \cong \pi_d(MT(d-r)) \cong \pi_1(B\Cob_d^{r}). 
\end{equation*}
This reduces the proof of Theorem \ref{vfcob.} to a study of $\pi_1(B\Cob_d^{r})$. Hence the rest of this paper will only be concerned with the fundamental group of $B\Cob_d^\theta$.

\section{Representing classes in $\pi_1(B\mathcal{C}_d^\theta)$ by closed manifolds}\label{reppi1}
In this section, we consider a general cobordism category corresponding to a fibration $\theta : X \to BO(d)$.
\begin{defi}
Let $(W,a)\in \Cob_d^\theta(M_0,M_1)$. The 1-simplex $\{(W,a)\}\times \Delta^1$ inside $B\Cob_d^\theta$ defines a path $\gamma_{(W,a)}$ between the points corresponding to $M_0$ and $M_1$. A concatenation of finitely many  such paths and their inverses (denoted $\bar{\gamma}_{(W,a)}$) will be called a zigzag of morphism paths.
\end{defi}
The goal of this section is to find conditions on the category $\Cob_d^\theta$ such that all elements of  $\pi_1(B\mathcal{C}_d^\theta)$ can be represented by a single morphism path $\gamma_{(W,a)}$ for some closed manifold $W$, considered as a morphism from the empty manifold to itself.
We start out by showing that elements of $\pi_1(B\Cob_d^\theta)$ are represented by zigzags.

First some notation. A path $\alpha : [0,1] \to \Ob(\Cob_d^\theta)$ that is smooth in the sense of \cite{monoids} and constant near the endpoints determines a morphism $(W_\alpha,1)$ such that $W_\alpha \subseteq  \R^{\infty-1} \times \R$ with $W_\alpha \cap (\R^{\infty-1} \times  \{t\})= \alpha(t)$ for all $t\in [0,1]$ and such that the projection $W_\alpha \to \{0\}\times \R$ is a submersion. We want to see that $\alpha \simeq \gamma_{W_\alpha}$ in $B\Cob_d^\theta$.

\begin{lem} \label{obsti}
The concatenation of a non-identity morphism path $\gamma_{(W,a)}$ and a smooth path $\alpha : I \to \Ob(\Cob_d^\theta)$ that is constant near  endpoints, is homotopic relative to endpoints inside $\Ob(\Cob_d^\theta) \cup (\Mor(\Cob_d^\theta) \times \Delta^1)$ to the morphism path $\gamma_{(W\circ W_\alpha,a+1)}$.
\end{lem}
\begin{proof}
Recall from \cite{monoids}, Theorem 3.9, that there are homotopy equivalences of categories
\begin{equation}\label{cathom}
\Cob_d^\theta \xleftarrow{c} D_\theta^\perp \xrightarrow{i} D_\theta.
\end{equation}
The precise definitions of the categories $D_\theta^\perp$ and $D_\theta$ are given in \cite{monoids}, Definition 3.8.

The morphism $(W \circ W_\alpha,a+1)$ corresponds to the morphism $(W\circ W_\alpha,0\leq a+1)$ in the category $ D_\theta$. In $\Mor(D_\theta)$,
\begin{equation*}
t\mapsto (W\circ W_\alpha,0\leq a+t)
\end{equation*}
is a path from $(W\circ W_\alpha,0\leq a)$ to $(W\circ W_\alpha,0\leq a+1)$ because all $a+t \in [a,a+1]$ are regular values for the projection $W\circ W_\alpha \to \{0\}\times \R$. 

This lifts to a path $(W_t,0\leq a+t)$ in $\Mor(D_\theta^\perp)$ under the inverse of the homotopy equivalence $i$. This just stretches $W\circ W_\alpha$ near $W\circ W_\alpha \cap ( \R^{\infty-1} \times \{a+t\})$ and leaves the rest fixed. In particular, $W_t\cap (\R^{\infty -1} \times \{a+t\})= \alpha(t)$.

Now, $c(W_t,0\leq a+t) = (W_t',a+t)$ defines a path from $(W,a)$ to $(W\circ W_\alpha,a+1)$ with  $W_t'\cap (\R^{\infty-1}\times \{a+t\})=\alpha(t) $. 
Thus $\gamma_{(W_t',a+t)}\cdot\alpha_{\mid [t,1]}$ is a homotopy from $\gamma_W\cdot \alpha$ to $\gamma_{W\circ W_\alpha}$.
\end{proof}

\begin{lem}\label{obmor}
A smooth path $\alpha : [0,1] \to \Ob(\Cob_d^\theta)$ that is constant near endpoints is homotopic relative to endpoints  to $\gamma_{W_\alpha}$ inside $B\Cob_d^\theta$.
\end{lem}

\begin{proof} 
Let $M=\alpha(0)$. Then $ W=M\times \R$ is a morphism. By Lemma \ref{obsti}, $\gamma_W\cdot \alpha$ is homotopic to $\gamma_{W\circ W_\alpha}$. There is a 2-simplex inside $B\Cob_d^\theta$ making $\gamma_{W\circ W_\alpha}$ homotopic to $\gamma_{W}\cdot \gamma_{ W_\alpha}$. Composing with $\bar{\gamma}_{W}$ proves the claim.
\end{proof}

\begin{thm}\label{morph}
Any path between two objects in $B\Cob_d^\theta$ is homotopic relative to endpoints to a zigzag of morphism paths.
\end{thm}

\begin{proof}
Let $f : [0,1] \to B\Cob_d^\theta $ be given such that $f(\{0,1\})\subseteq \Ob(\Cob_d^\theta)$. First we deform $f$ relative to the endpoints to have image in $\Ob(\Cob_d^\theta) \cup (\Mor(\Cob_d^\theta) \times \Delta^1)$.
Such an  $f$ may again be deformed to a composition of finitely many morphism paths, inverse morphism paths, and paths in the object space. 

By \cite{monoids}, Lemma 2.18, each path $\alpha: I \to \Ob(\Cob_d^\theta)$ is homotopic to a smooth path $\alpha'$ that is constant near endpoints. By Lemma \ref{obmor}, this is homotopic to $\gamma_{W_{\alpha'}}$.
\end{proof}

We are now ready to give conditions under which every element of $\pi_1(B\Cob_d^\theta)$ may be represented by a single morphism path.

\begin{thm} \label{morsti}
Assume:
\begin{itemize}
\item[(i)] Any morphism $W \in \Cob_d^\theta(\emptyset,\emptyset)$ has an inverse $W^- \in \Cob_d^\theta(\emptyset,\emptyset)$ such that the disjoint union $W \circ W^-\in \Cob_d^\theta(\emptyset,\emptyset)$ defines a null-homotopic loop in $B\Cob_d^\theta$.
\item[(ii)] If $W \in \Cob_d^\theta(M_0,M_1)$, then there exists a morphism $\overline{W} \in \Cob_d^\theta(M_1,M_0)$ in the opposite direction.
\end{itemize}
In this case, any element of $\pi_1(B\Cob_d^\theta)$ can be represented by a morphism path $\gamma_{W}$ for some closed $\theta$-manifold $W \in \Cob_d^\theta(\emptyset,\emptyset)$.
\end{thm}
From Theorem \ref{morph} we know that we can always represent an element of $\pi_1(B\Cob_d^\theta)$ as a zigzag of morphisms.
In general, the closed manifold cannot be chosen diffeomorphic to the one defined by glueing together the underlying manifolds in the zigzag, since this may not allow a $\theta $-structure.
\begin{proof}
Let $\gamma : I \to B\Cob_d^\theta$ be a path from the empty manifold to itself. By Theorem~\ref{morph}, we can assume that $\gamma $ is represented by a zigzag.

For a pair of composable morphisms $(W_1,W_2)\in N_2(\Cob_d^\theta)$,
 the corresponding 2-simplex $\{(W_1 , W_2)\}\times \Delta^2$ inside $B\Cob_d^\theta$ defines a homotopy 
\begin{equation} \label{glue}
\gamma_{W_1}\cdot \gamma_{W_2} \simeq \gamma_{W_1 \circ W_2}. 
\end{equation}
Thus we may assume that $\gamma$ is an alternating zigzag of morphism paths 
\begin{equation*}
\gamma = \gamma_{W_1}\cdot\bar{\gamma}_{W_2}\cdot \gamma_{W_{3}} \dotsm \bar{\gamma}_{W_n}.
\end{equation*}
 Of course, it could also happen that the first path is an inverse path or that $n$ is odd. These cases are similar.

For each  $i$, choose an opposite $\overline{W}_i$ of $W_i$, guaranteed by assumption (ii). Then 
\begin{align*}
\gamma =& \gamma_{W_1}\cdot \bar{\gamma}_{W_2}\cdot \gamma_{W_{3}} \dotsm \bar{\gamma}_{W_n}\\
\simeq &\gamma_{W_1} \cdot {({\gamma}_{\overline{W}_2}\cdot \gamma_{W_{3}} \dotsm {\gamma}_{\overline{W}_n})} \cdot \overline{({\gamma}_{\overline{W}_2}\cdot \gamma_{W_{3}} \dotsm {\gamma}_{\overline{W}_n})}\cdot \bar{\gamma}_{W_2} \cdot (\bar{\gamma}_{\overline{W}_{3}} \dotsm \bar{\gamma}_{W_n}) \\ 
&\cdot \overline{(\bar{\gamma}_{\overline{W}_{3}} \dotsm \bar{\gamma}_{W_n})} \dotsm {\gamma}_{W_{n-1}} \cdot(\gamma_{\overline{W}_n}) \cdot \overline{(\gamma_{\overline{W}_n})} \cdot \bar{\gamma}_{{W}_n}\\
=&(\gamma_{W_1}\cdot{\gamma}_{\overline{W}_2}\cdot\gamma_{W_{3}} \dotsm {\gamma}_{\overline{W}_n})\cdot (\bar{\gamma}_{\overline{W}_n}\dotsm \bar{\gamma}_{W_{3}}\cdot \bar{\gamma}_{\overline{W}_2} \cdot \bar{\gamma}_{W_2} \cdot \bar{\gamma}_{\overline{W}_{3}} \dotsm \bar{\gamma}_{W_n}) \\ 
&\cdot({\gamma}_{{W}_{n}} \dotsm {\gamma}_{\overline{W}_3} \cdot {\gamma}_{{W}_3} \dotsm {\gamma}_{\overline{W}_n}) \dotsm 
( {\gamma}_{{W}_n}\cdot {\gamma}_{\overline{W}_{n-1}}\cdot {\gamma}_{W_{n-1}}\cdot \gamma_{\overline{W}_n})\cdot (\bar{\gamma}_{\overline{W}_n}\cdot \bar{\gamma}_{{W}_n})\\
\simeq &(\gamma_{W_1\circ \overline{W}_2 \circ W_{3}\circ \dotsm \circ \overline{W}_n})\cdot (\bar{\gamma}_{{W}_n\circ \dotsm \circ \overline{W}_{3} \circ {{W}_2} \circ {\overline{W}_2} \circ {{W}_{3}}\circ \dotsm \circ {\overline{W}_n}})\\
&\cdot( {\gamma}_{{W}_{n}\circ \dotsm \circ {\overline{W}_3} \circ {{W}_3}\circ  \dotsm \circ {\overline{W}_n}}) \dotsm
( {\gamma}_{{{W}_n}\circ {\overline{W}_{n-1}} \circ W_{n-1} \circ \overline{W}_n}) \cdot (\bar{\gamma}_{{W}_n \circ \overline{W}_n}).
\end{align*} 
The idea is here that we first run along $\gamma_{W_1}$ as we are supposed to. Then we follow morphism paths in the positive direction all the way to the empty manifold and go back again, now following paths in the inverse direction. We run $\gamma_{W_2}$ backwards as we are supposed to and then follow paths in the inverse direction back to the base point and go back again. Continuing this way, we end up with a path that is homotopic to the original one. But this new path has the property that we always run from the base point to itself along paths that are either all positively directed or all negatively directed. Thus we may glue the manifolds together by \eqref{glue}. 

This yields an expression for $\gamma $ involving only paths of closed manifolds. We now apply assumption (i) to all the paths that are still travelled in the wrong direction. Finally we apply \eqref{glue} again to write $\gamma $ as a path corresponding to a single morphism:
\begin{align*}
\gamma \simeq & \gamma_{(W_1\circ \overline{W}_2 \circ W_{3}\circ \dotsm \circ \overline{W}_n)}\cdot {\gamma}_{({W}_n\circ \dotsm \circ {\overline{W}_{3}} \circ {{W}_2} \circ {\overline{W}_2} \circ {{W}_{3}}\circ \dotsm \circ {\overline{W}_n})^-}\\
&\cdot {\gamma}_{({W}_{n}\circ \dotsm \circ {\overline{W}_3} \circ {{W}_3}\circ \dotsm \circ {\overline{W}_n})} \dotsm 
 {\gamma}_{({{W}_n}\circ{\overline{W}_{n-1}} \circ W_{n-1} \circ \overline{W}_n)}\cdot {\gamma}_{({W}_n \circ \overline{W}_n)^-}\\
\simeq &\gamma_{(W_1\circ \overline{W}_2 \circ W_{3}\circ \dotsm \circ \overline{W}_n)\circ({W}_n\circ \dotsm \circ \overline{W}_{3} \circ {{W}_2} \circ {\overline{W}_2} \circ {{W}_{3}}\circ \dotsm \circ {\overline{W}_n})^-} \\
&_{\circ {({W}_{n}\circ \dotsm \circ {\overline{W}_3} \circ {{W}_3} \circ\dotsm \circ {\overline{W}_n})} \circ  \dotsm
{ \circ ({{{W}_n}\circ {\overline{W}_{n-1}} \circ W_{n-1} \circ \overline{W}_n}) \circ ({W}_n \circ \overline{W}_n)^-}.}
\end{align*}
\end{proof}

\section{Geometric interpretation of the invariants}\label{interprinv}
We now return to the cobordism category with vector fields. We start out by showing that $\Cob_d^r$ satisfies the conditions (i) and (ii) of Theorem \ref{morsti}. 
From this we obtain the geometric interpretations of the homotopy groups $\pi_d(MT(d-r))\cong \pi_{r+1}(B\Cob_d)$.

\begin{thm}\label{vende}
Let $d$ be odd or $r < \frac{d}{2}$. Let $W \in \Cob_d^r(M_0,M_1)$. Then there exists a $\overline{W} \in \Cob_d^r(M_1,M_0)$.
\end{thm}

\begin{proof}
Suppose $(W,a) \in \Cob_d^r(M_0,M_1)$ is given. That is, $W\subseteq (-1,1)^n \times \R$ is a $d$-dimensional manifold with a section $v: W \to V_r(TW)$. The reflection $t \mapsto a-t$ in the $(n+1)$th coordinate takes $W$ to a morphism in $\Cob_d(M_1,M_0)$. In the oriented case, the orientation must be reversed. However, it does not define an element of $\Cob_d^r(M_1,M_0)$ yet, since the vector fields on $M_0$ and $M_1$ have been reflected in the normal direction. They must be reflected once more to get the correct vector fields on the objects. We need a way to extend these reflected  vector fields to $W_0=W\cap (-1,1)^n \times [0,a]$. 

For this, choose a normal vector field on $ \partial W_0$ and extend this to a unit vector field
\begin{equation*}
V : W^{(d-1)}_0 \to TW_{\mid W_0^{(d-1)}}
\end{equation*}
on the $(d-1)$-skeleton $W_0^{(d-1)}$. This is always possible by standard obstruction theory, see \cite{steenrod}.
This defines a map
\begin{equation*}
\sigma_V: W^{(d-1)}_0 \to O(TW)_{\mid W^{(d-1)}_0}.
\end{equation*}
Here $O(TW)$ is the bundle over $W$ with fiber over $x$ the orthogonal group $O(T_xW)$, and $\sigma_V(x)$ is defined to be the reflection of $T_xW$ that takes $V(x)$ to $-V(x)$ and leaves $V(x)^\perp$ fixed.

With this definition, $\sigma_V$ acts on the given vector fields by multiplication
\begin{equation*}
\sigma_V(x)\cdot  v(x) = w(x): W^{(d-1)}_0 \to V_r(TW_{\mid W^{(d-1)}_0})
\end{equation*}
for all $x\in W^{(d-1)}_0$. On $\partial W_0$, $\sigma_V$ is the reflection of the normal direction, so $w$ is an extension of the reflected vector fields on $\partial W_0$ to the $(d-1)$-skeleton.

It remains to extend $w$ over each $d$-cell $D\subseteq W_0$. This may not be possible. The idea is to take the connected sum of $W_0$ and a suitable manifold in the interior of $D$ such that $w$ extends to the glued in manifold. 

Choose a trivialization $TW_{\mid D} \cong D\times \R^d$. We may assume that $v_{\mid \partial D}:S^{d-1} \to V_{d,r}$ is constant equal to $v_0$ since it extends to $D$. Then $w$ is given on $\partial D$ by
\begin{equation*}
w_{\mid \partial D}: S^{d-1} \xrightarrow{\sigma_V} O(d) \xrightarrow{h} V_{d,r},
\end{equation*}
where $h$ is evaluation on $v_0$.

In the trivialization, $V_{\mid \partial D}$ is a map $V_{\mid \partial D} : S^{d-1}  \to S^{d-1}$. Let $\rho_d : S^{d-1} \to O(d)$ be the map that takes $x\in S^{d-1}$ to the reflection of the line spanned by $x$. Then 
\begin{equation*}
\sigma_V = \rho_d \circ V_{\mid \partial D} \textrm{ and }
w_{\mid \partial D} = h \circ \rho_d \circ V_{\mid \partial D}.
\end{equation*}
By possibly dividing $D$ into smaller cells, we may assume that the degree of $V_{\mid \partial D}$ is either 0 or $\pm 1$.

If the degree of $V_{\mid \partial D}$ is zero, then $\sigma_V$ is homotopic to a constant map. Hence, so is $w_{\mid \partial D}$, and the vector fields extend to all of $D$.

If the degree is $+1$, then $V$ is homotopic to the identity map. This means that $\sigma_V$ is the reflection in the normal direction. Choose $r$ independent vector fields on the torus $T^d$. Cut out a disk $D'$. The vector fields on $\partial D'$ are now homotopic to $v_{\mid \partial (W\backslash D)}$, since both extend over a disk. Thus, after reflecting $v_{\mid \partial D}$ in the normal direction, we can form the connected sum of $W$ and $T^d$ in the interior of $D$ with $r$ independent vector fields extending $w_{\mid \partial(W\backslash D)}$ over $T^d\backslash D'$.

If the degree of $V_{\mid \partial D}$ is $-1$ and $d$ is odd, $V$ is homotopic to minus the identity. Since $\sigma_V=\sigma_{-V}$, we may do as in the degree $+1$ case.

We are now left with the case where $d=2k$ is even and the degree is $-1$. In this case we would like to take the connected sum with a product of two spheres, rather than a torus. 

First look at what happens to the vector fields when they are reflected in a map of degree $-1$. Consider the diagram where the middle vertical sequence is exact and $p$ is the map that forgets the first $r-1$ vectors: 
\begin{equation}\label{Odia}
\begin{vcenter}{\xymatrix{{\pi_{d-1}(S^{d-1})}\ar[d]^{V_{\partial D}}&{\pi_{d}(S^{d})}\ar[d]^{\delta}\ar[dr]_{\delta_d}\ar[drr]^{\cdot 2}&{}&{}\\
{\pi_{d-1}(S^{d-1})}\ar[r]^-{\rho_d}\ar[d]&{\pi_{d-1}(O(d))}\ar[r]^{h}\ar[d]&{\pi_{d-1}(V_{d,r})}\ar[r]^-{p}&{\pi_{d-1}(S^{d-1})}\\
{\pi_{d-1}(S^{d})=0}\ar[r]^-{\rho_{d+1}}&{\pi_{d-1}(O(d+1)).}&{}&{}
}}\end{vcenter}
\end{equation}
It follows that $\rho_d$ maps into the image of $\delta$. The composition $p\circ h \circ \rho_d$ maps $x\in S^{d-1}$ to the reflection of a fixed vector in the $x$-direction. This is the obstruction to a zero-free vector field on $S^{d}$, hence the degree is the Euler characteristic $\chi(S^d)=2$, see \cite{steenrod}. Thus $\rho_d([1])=\delta([1])$, and therefore 
\begin{equation*}
[w_{\mid \partial D}] = h \circ \rho_d ([-1]) = \delta_d([-1]). 
\end{equation*}
Here and in the following, $[m]\in \pi_l(S^l)$ denotes the class of degree $m$ maps. 

Consider
\begin{alignat*}{2}
&S^k\times S^k &\quad &\textrm{ if $k$ is even,}\\
&S^{k-1}\times S^{k+1} &\quad &\textrm{ if $k$ is odd.}
\end{alignat*}  
For simplicity, we write this product as $S^i \times S^j$ in the following. 
Choose $r$ vector fields with one singularity on each sphere. This is possible because $r<k$ by assumption. 
We may assume that these vector fields are given outside small open disks $D^i$ and $D^j$ by
\begin{equation*}
\begin{split}
u_1 &: S^i\backslash D^i \to V_r(TS^i)\\
u_2 &: S^j \backslash D^j \to V_r(TS^j),
\end{split}
\end{equation*} 
respectively. 
This defines $r$ vector fields on $S^i\times S^j$ with one singularity inside $D^i \times D^j$ via the formula
\begin{equation*}
u(x,y) = \frac{u_1(x)+u_2(y)}{\sqrt{2}}
\end{equation*}
on $S^i\backslash D^i\times S^j\backslash D^j$ and
\begin{equation*}
u(x,y) = \frac{1}{\sqrt{2}}\left({|x|u_1\left(\frac{x}{|x|}\right)+|y|u_2\left(\frac{y}{|y|}\right)}\right)
\end{equation*}
on $D^{i}\times S^j \cup S^{i}\times D^j$. On the boundary of $D^{i}\times D^j$, this is the join 
\begin{equation*}
u_{1\mid S^{i-1}} \star u_{2\mid S^{j-1}} : S^{i-1} \star S^{j-1} \to V_{d,r}. 
\end{equation*}
This map represents the obstruction to $r$ independent vector fields on $S^i\times S^j$. 

Now look at the diagram
\begin{equation*}
\xymatrix{{\pi_i(S^i)\times \pi_j(S^j)}\ar[d]^{\delta_i\times \delta_j}&{\pi_d(S^d)}\ar[d]^{\delta_d}\ar[dr]^{\cdot 2}&{}\\
{\pi_{i-1}(V_{i,r})\times\pi_{j-1}(V_{j,r})}\ar[r]^-{\star}\ar[d]^{\eta_i \times \eta_j }&{\pi_{d-1}(V_{d,r})}\ar[r]^-{p}\ar[d]^{\eta_d}&{\pi_{d-1}(S^{d-1})}\\
{\pi_{i-1}(V_{i+1,r+1})\times\pi_{j-1}(V_{j+1,r+1})}&{\pi_{d-1}(V_{d+1,r+1}).}&{}
}
\end{equation*}
For $l=i,j,d$, the class $\delta_l([1])$ is the obstruction to $r$ independent vector fields on $S^l$. Thus the homotopy class of $u_{1\mid S^{i-1}} \star u_{2 \mid S^{j-1}}$ is $\delta_i([1])\star \delta_j([1])$. By \cite{james1}, formula (2.12 b), 
\begin{equation*}
\eta_d(\delta_i([1])\star \delta_j([1])) = \eta_i(\delta_i([1])) \star \delta_j'([1]).
\end{equation*}
Here $\delta_j':\pi_j(S^j) \to \pi_{j-1}(V_{j,r+1})$ is the boundary map. This is well-defined because $r < j $. 
But $\eta_i \circ \delta_i = 0$, so $\delta_i([1])\star \delta_j([1])$ is in the image of $\delta_d$. Furthermore,
\begin{equation*}
p(\delta_i([1])\star \delta_j([1]))=[4]
\end{equation*}
since it is the obstruction to a single zero-free vector field on $S^i\times S^j$, which is $\chi(S^i\times S^j)=4$. Thus $\delta_i([1])\star \delta_j([1]) =\delta_d([2])$.

Summarizing the above, there are independent vector fields on $S^i\times S^j \backslash D^i\times D^j$ and $\partial(W\backslash \indre(D))$ given on the boundaries of the removed disks by $\delta_d([2])$ and $\delta_d([-1])$, respectively. We can take the connected sum if the vector fields agree after reflecting the ones on $\partial (D^i\times D^j)$. That is, it remains to show $\rho_d \cdot \delta_d([2]) = \delta_d([-1])$.

Note that
\begin{equation}\label{degref}
p(\rho_d\cdot \delta_d([m])) = \rho_d\cdot p(\delta_d([m])) =\rho_d\cdot [2m] = [2-2m].
\end{equation} 
The last equality follows because a degree $2m$ map $S^{d-1} \to S^{d-1}$ defines a vector field on $S^d$ with two singularities, one of degree $2m$ and one of  degree $\rho_d\cdot [2m]$. The sum of these must be $\chi(S^d)=2$ by obstruction theory.

Also note that 
\begin{equation*}
\eta_d(\rho_d\cdot \delta_d([2])) = \rho_d'\cdot (\eta_d \circ \delta_d([2]))=0
\end{equation*}
 where $\rho_d'$ is the composition of $\rho_d$ with $O(d) \to O(d+1)$ and hence trivial by \eqref{Odia}. 
Thus $\rho_d\cdot \delta_d([2])$ is in the image of $\delta_d$. According to \eqref{degref},
\begin{equation*}
p(\rho_d\cdot\delta_d([2])) = p(\delta_d([-1])) =[-2]. 
\end{equation*}
But $p\circ \delta_d$ is injective, so $\rho_d\cdot \delta_d([2])=\delta_d([-1])$ as claimed.
\end{proof}

This proof was inspired by Proposition 4.23 in \cite{monoids}.

\begin{rem}
In the case where $S^d$ allows $r$ independent vector fields, see \cite{adams1}, it is possible to choose $\overline{W}$ diffeomorphic to $W$. This is because we may glue in disks, rather than tori, when constructing $\overline{W}$.  

If $S^d$ does not allow $r$ independent vector fields, the disk $D^d$ is an example of a morphism such that $\overline{D}^d$ cannot be chosen diffeomorphic to $D^d$. Otherwise they would glue together to a sphere with $r$ independent vector fields.

The above proof would work more generally for any $\theta $-structure satisfying that $S^{d-1}$ with any $\theta$-structure bounds a $\theta $-manifold. For $d$ odd, it would also suffice that $S^d$ allows a $\theta$-structure. 
\end{rem}

\begin{ex}
Consider the case $d=2$ and $r=1$. A surface of genus $g>1$ allows a vector field with only one singularity. Cut out a disk containing the singularity. This defines a morphism from $\emptyset $ to $S^1$. If there were a morphism in the opposite direction, they would glue together to a closed surface of genus at least $g>1$ with a zero-free vector field, which is impossible.

This shows that the condition  $r<\frac{d}{2}$ is not always redundant. We do not know whether it is best possible. 
\end{ex}

\begin{thm}\label{antalvf}
There are weak homotopy equivalences
\begin{equation*}
B\Cob _{d+k}^{r+k} \to \Omega^k B\Cob_d^r \to \Omega^{\infty+d+k-1}MT(d-r).
\end{equation*}
In the case $k=1$, assume $M\in \Ob(\Cob _{d+1}^{r+1})$ with $r$ tangent vector fields and the $(r+1)$th vector field equal to the positively directed normal $\eps$. Then the component of $M$ in $\pi_0(B\Cob _{d+1}^{r+1})$ is mapped to the morphism path in $\pi_1(B\Cob_d^r)$ corresponding to $M$, now considered as a morphism in $\Cob_d^r(\emptyset,\emptyset)$ with the $r$ tangent vector fields.  Both correspond to the Pontryagin--Thom element in $\pi_d(MT(d-r))$.
\end{thm}

\begin{proof}
We have the following commutative diagram for each $n$
\begin{equation}\label{psi}
\vcenter{\xymatrix{{\Omega^k B\Cob_{d,n}^r}\ar[r]&{\Omega^k \psi_d^r(n,1)}\ar[d]\ar[r]&{\Omega^{n+k-1} \psi_d^r(n,n)}\ar[d]\\
{B\Cob _{d+k,n+k}^{r+k}}\ar[r]&{\Omega^k \psi_{d+k}^{r+k}(n+k,k+1)}\ar[r]&{\Omega^{n+k-1} \psi_{d+k}^{r+k}(n+k,n+k).}
}}
\end{equation}
The vertical maps take a manifold $M \subseteq \R^{n}$ with $r$ vector fields to $M\times \R^k \subseteq \R^{n}\times \R^k$ with the $r$ vector fields from $M$ together with the $k$ standard vector fields in the $\R^k$ direction. The horizontal maps are the homotopy equivalences from \cite{monoids}.

The diagram
\begin{equation}\label{diagphi}
\vcenter{\xymatrix{{\psi_d^r(n,n)}\ar[d]\ar[r]&{\Th(i_r^*U_{d,n-d}^\perp \to V_r(U_{d,n-d}))}\ar[d]\\
{\psi_{d+k}^{r+k}(n+k,n+k)}\ar[r]&{\Th(i_{r+k}^*U_{d+k,n-d}^\perp \to V_{r+k}(U_{d+k,n-d}))}
}}
\end{equation}
also commutes. The horizontal maps take a manifold $M$ to $-p$ in the fiber over $T_pM$ with the vector fields evaluated at this point, where $p$ is the point on $M$ closest to the identity (whenever this is defined).  Hence commutativity follows. Some details have been omitted, see \cite{monoids} for the precise definition of the horizontal maps.

If we let $n$ tend to infinity, the right vertical map in \eqref{psi} is a homotopy equivalence, since the diagram
\begin{equation*}
\xymatrix{
{\Th(U_{d-r,n-d}^\perp \to G(d-r,n-d))}\ar[r]\ar[dr]&{\Th(i_r^*U_{d,n-d}^\perp \to V_r(U_{d,n-d}))}\ar[d]\\
{}&{\Th(i^*_{r+k}U_{d+k,n-d}^\perp \to V_{r+k}(U_{d+k,n-d}))}
}
\end{equation*} 
commutes and the two maps to the left are homotopy equivalences of spectra.

Now let $k=1$ and let $M\subseteq (-1,1)^{n}$  be as in the theorem.  This corresponds to the manifold $M\times \R$ in $\psi_{d+1}^{r+1}(n+1,1)$ with the induced vector fields. The first lower horizontal map in \eqref{psi} takes this to a loop $\R^+ \to \psi_{d+1}^{r+1}(n+1,2)$ given by $t \mapsto M\times \R -(0,\dots , 0,t,0)$. But this is the image of the loop $ \R^+ \to \psi_{d}^{r}(n,1)$ given by $t \mapsto M-(0,\dots , 0 ,t)$ under the vertical map. This is again homotopic to the image of the morphism path for $M$ in $B\Cob_{d,n}^r$ under the map $\Omega B\Cob_{d,n}^r \to \Omega \psi_{d}^{r}(n,1)$, as one may see by checking the definitions. 

%
%
Going to the right in the upper part of \eqref{psi} shows that $M$ corresponds to the map $(\R^n)^+ \to \psi(n,n)$ given by $t\mapsto M-t$. Going through the definition of the horizontal maps in \eqref{diagphi}, one realizes that this corresponds to the Pontryagin--Thom maps.
\end{proof}

\begin{thm}\label{inverse}
For $r\geq 0$, all morphisms in $\Cob_d^r(\emptyset , \emptyset )$ have inverses in the sense of Theorem \ref{morsti} (i).
\end{thm}

\begin{proof}
Let $W$ be a closed $d$-dimensional manifold with $r$ orthonormal vector fields 
\begin{equation*}
v_1,\dots ,v_r : W \to TW. 
\end{equation*}

Assume $d$ is odd or $r\geq 1$. We consider $W$ as an object of $\Cob_{d+1}^{r+1}$ with the positive normal vector  field $\eps$ as the $(r+1)$th vector field. Then there are $r+1$ vector fields on $W\times \R$, given on $W\times [0,1]$ as follows:
Choose a zero-free vector field $v:W \to TW$. If $r\geq 1$, we simply choose this to be $v_r$. If $d$ is odd and $r=0$, we may choose $v$ arbitrary.
Define $r+1$ orthonormal vector fields $w_1,\dots , w_{r+1}$ on $W\times [0,1]$ by
\begin{align*}
w_i(x,t) =& v_i(x) \\
w_r(x,t) =& \cos (\pi t)v(x) + \sin(\pi t)\eps (x)\\
w_{r+1}(x,t) =& -\sin(\pi t)v(x) + \cos(\pi t)\eps (x).
\end{align*} 
Extend these trivially to $W\times \R$. Embed $W\times \R$ as a cobordism from $W \times \{0,1\}$ to $\emptyset$ in $\Cob_{d+1}^{r+1}$. Thus $W \times \{0,1\}$ belongs to the base point component of $B\Cob_{d+1}^{r+1}$. 
Let $W^- $ be $ W\times \{1\}$ with the induced vector fields and, in the oriented case, orientation.
 
Under the isomorphism from Theorem \ref{antalvf} 
\begin{equation*}
\pi_1(B\Cob_{d}^{r})\to \pi_0(B\Cob_{d+1}^{r+1}), 
\end{equation*}
$W \times \{0,1\}$ lifts to $\gamma_{W\circ W^-}$, so this must be null-homotopic.

In the remaining case where $d$ is even and $r=0$, we may still view $W$ as an object in $\Cob_{d+1}^{1}$ with vector field $\eps$. As before, we seek another manifold such that the disjoint union with $W$ bounds a manifold with a zero-free vector field extending the inward normal. By \cite{reinhart} it is enough to find a manifold $W^-$ which is a cobordism inverse to $W$ and has Euler characteristic 
\begin{equation}\label{chi}
\chi(W^-) = -\chi(W).
\end{equation} 

Let $W'$ be a copy of $W$. In the oriented category, give it the opposite orientation. Then $W'$ is a cobordism inverse of $W$. 
Taking the disjoint union with a sphere increases the Euler characteristic by 2, and taking disjoint union with a connected sum of two tori decreases the Euler characteristic by 2. Thus, defining $W^-$ to be the disjoint union of $W'$ and a suitable bounding manifold, $\eqref{chi}$ is satisfied.
\end{proof}

\begin{cor}\label{path}
For $d$ odd or $r<\frac{d}{2}$, any class in $\pi_1(B\Cob_d^r)$ may be represented by a morphism path.
\end{cor}
\begin{proof}
This follows from Theorem \ref{morsti}, \ref{vende} and \ref{inverse}.
\end{proof}

\begin{defi}
Let $M_0, M_1 \in \Ob(\Cob_d^r)$. We say that $M_0$ is vector field cobordant to $M_1$ if $\Cob_d^r(M_0,M_1)$ is non-empty.
\end{defi}
We are now ready to prove Theorem \ref{vfcob.}.
%
%

\begin{proof}[Proof of Theorem \ref{vfcob.}]
For the equivalence relation, symmetry follows from Theorem~\ref{vende} and transitivity is given by composing morphisms.
 
By Theorem \ref{antalvf}, $ \pi_r(B\Cob_{d-r}) \cong \pi_0(B\Cob_d^r) $.
If two manifolds are vector field cobordant, they obviously belong to the same path component of $B\Cob_d^r$. Conversely, if $M_0$ and $M_1$ belong to the same path component, there is a zigzag of morphisms relating them by Theorem \ref{morph}. Thus they are vector field cobordant. 

By Theorem \ref{path}, any element of $\pi_r(B\Cob_{d-r})\cong \pi_1(B\Cob_{d-1}^{r-1})$ is represented by a morphism for $d-1$ odd or $r < \frac{d}{2}$. This corresponds to an object of $\Cob_d^r$ with the $r$th vector field equal to $\eps$ by Theorem \ref{antalvf}.
\end{proof}

In particular, $\pi_d(MT(d)) \cong \pi_0(B\Cob_{d+1}^1)$ is the group of Reinhart cobordism classes of $d$-dimensional manifolds. 
A Reinhart cobordism from $M^d$ to $N^d$ is a cobordism with a zero-free vector field which is inward normal at $M$ and outward normal at $N$. The equivalence classes are determined in \cite{reinhart}. 

\begin{cor}\label{corres}
Let $d$ odd or $r < \frac{d}{2}$. The image of $\pi_d(MT(d-r)) \to \pi_d(MT(d))$ is the group of Reinhart cobordism classes containing a manifold that allows $r$ independent tangent vector fields.
\end{cor}

\begin{proof}
Let $\beta$ be a Reinhart cobordism class. If $\beta $ lifts to $\alpha \in \pi_{d}(MT(d-r))$, this is represented by a morphism loop in $\pi_1(B\Cob_d^r)$ by Corollary \ref{path}. This morphism is a closed $d$-dimensional manifold with $r$ independent tangent vector fields and it represents $\beta$ by \eqref{forget}. 
\end{proof}

A diagram similar to \eqref{forget} yields an interpretation of the maps
\begin{equation*}
\pi_{d-1}(MT(d-r-k))\to \pi_{d-1}(MT(d-r)):
\end{equation*}

\begin{cor}
Suppose $d$ is odd or $r<\frac{d}{2}$. Under $\pi_0(B\Cob_d^{r+k})\to \pi_0(B\Cob_d^{r})$, a component containing a manifold $M$ with $r$ orthonormal sections in $TM\oplus \R$ is in the image if and only if there is a cobordism with $r$ orthonormal vector fields from $M$ to some $M'$ such that the $r$ sections in $TM'\oplus \R$ extend to $r+k$ orthonormal sections.

If $d$ is even or $r<\frac{d-1}{2}$, the image of $\pi_0(B\Cob_d^{r+1})\to \pi_0(B\Cob_d^{r})$ is the subgroup of the vector field cobordism group containing all manifolds with $r$ orthonormal tangent vector fields.
\end{cor}

\section{Equivalence of zigzags}\label{genrel}
We saw in Theorem \ref{morph} that all elements of $\pi_1(B\Cob_d^\theta)$ are represented by zigzags of morphism paths. In this section we find necessary and sufficient conditions for two such zigzags to be homotopic.

First some notation. We picture a zigzag of morphisms $\dotsm \bar{\gamma}_{W_i}\cdot \gamma_{W_{i+1}}\dotsm$ by
\begin{equation}\label{zzrepr}
\dotsm \to M_i \xleftarrow{W_i} M_{i+1} \xrightarrow{W_{i+1}} M_{i+2} \xleftarrow{} \dotsm
\end{equation}
Moreover, $\partial_i : \Mor(\Cob_d^\theta) \to \Ob(\Cob_d^\theta)$ will denote the boundary maps which are defined for $W \in \Cob_d^\theta(M_0,M_1)$ by $\partial_i(W)=M_i$ for $i=0,1$.

The main theorem of this section is:
\begin{thm}\label{zzrel}
Two zigzags  represented by a diagram like \eqref{zzrepr} are homotopic relative to endpoints if and only if they are related by a finite sequence of moves of the following two types:
\begin{enumerate}[label=(\Roman*)]
\item \label{rel1} Any sequence of arrows from $M_i$ to $M_j$ in the diagram
\begin{equation*}
\xymatrix{{M_0}\ar[d]_{W_1}\ar[dr]^{W_1\circ W_2}&{}\\
{M_1}\ar[r]_{W_2}&{M_2,}
}
\end{equation*}
may be replaced by any other such sequence.
\item \label{rel2}
Suppose $W,W'\in \Cob_d^\theta(M_0,M_1)$. Then $W$ may be replaced by  $W'$ in the zigzag
if there exists a path $\gamma: I \to \Mor(\Cob_d^\theta)$ from $W$ to $W'$ such that  $\partial_i\circ \gamma : I \to \Ob(\Cob_d^\theta)$ are constant for $i = 0,1$.
\end{enumerate}
\end{thm}

A more geometric interpretation of the relation \ref{rel2} is given by the following:
\begin{lem} \label{diffeo}
Two morphisms $W_0$ and $W_1$ may be joined by a path $\gamma : I \to \Mor(\Cob_d^\theta)$ with $\partial_i\circ\gamma$ constant if and only if there is a diffeomorphism between them that fixes $\partial_i W_j$ pointwise for $i,j =0,1$ and preserves the equivalence class of $\theta$-structures.
\end{lem}
%
 An equivalence class of $\theta $-structures means an element of $\pi_0(\textrm{Bun}(TW,\theta^*U_d))$ where $\textrm{Bun}(TW,\theta^*U_d)$ is the space of bundle maps $TW \to \theta^*U_d$. The proof is straightforward from the way the topology on $\Mor(\Cob_d^\theta)$ is defined.

\begin{lem}\label{diffrel}
Let $\gamma :[0,1] \to \Mor(\Cob_d^\theta)$ be a smooth path from $W_0$ to $W_1$ that is constant near $0,1$. Let $W_{\partial_i\gamma}$ be the morphisms determined by $\partial_i\circ \gamma:[0,1]\to \Ob(\Cob_d^\theta)$ for $i=0,1$. Then 
$\gamma_{W_0\circ W_{\partial_1\gamma}}$ and $\gamma_{W_{\partial_0\gamma}\circ W_1}$  differ only by a Type \ref{rel2} move.
\end{lem}

\begin{proof}
Look at the morphisms $(W_{\partial_0\gamma},0\leq 1)$ and $(W_{\partial_1 \gamma },0\leq 1)$ in the category $D^\theta_d$. There are paths in $\Mor(D^\theta_d)$ given by $ (W_{\partial_0\gamma},0\leq t)$ and $(W_{\partial_1\gamma},t\leq 1)$ for $t\in (0,1)$. These lift to paths $\gamma_0$ and $\gamma_1$ in $\Mor(\Cob_d^\theta)$ under the homotopy equivalences \eqref{cathom} satisfying
\begin{equation*}
\begin{split}
\gamma_0(t) &\in \Cob_d^\theta(M_0,\partial_0(\gamma(t))) \\ 
\gamma_1(t) &\in \Cob_d^\theta(\partial_1(\gamma(t)),M_1). 
\end{split}
\end{equation*}
Thus the composition of morphisms $\gamma_0(t)\circ \gamma(t) \circ \gamma_1(t)\in \Cob_d^\theta(M_0,M_1)$ is a well-defined path in the morphism space for $t\in (0,1)$. This naturally extends to all $t \in [0,1]$, and this is the desired path from ${W_0} \circ {W_{\partial_1\gamma}}$ to ${W_{\partial_0\gamma}} \circ {W_1}$.
\end{proof}

The next two proofs consider homotopy groups with multiple base points. If $X$ is a topological space and $X_0$ is a discrete subset, $\pi_1(X,X_0)$ denotes the set of homotopy classes of paths in $X$ starting and ending in $X_0$. The path composition makes this into a groupoid where the identity elements correspond to the constant paths. 

\begin{thm}\label{zz0}
Any two zigzags that are homotopic inside $\Ob(\Cob_d^\theta)\cup (\Mor(\Cob_d^\theta)\times \Delta^1)$ are related by a sequence of Type \ref{rel1} and \ref{rel2} moves.
\end{thm}

\begin{proof}
First choose a set of objects $M_i$ for $i \in I$, one in each path component of $\Ob(\Cob_d^\theta)$. Then choose a $W_j$, $j\in J$, in each component of $\Mor(\Cob_d^\theta)$ such that $\partial_\eps(W_j) \in \{M_i, i \in I\}$ for all $j\in J$ and $\eps = 0,1$. These will serve as the base point sets. By construction, the source and target maps are base point preserving.

To describe $\pi_1(\Ob(\Cob_d^\theta)\cup (\Mor(\Cob_d^\theta)\times \Delta^1))$, we need a generalized version of the van Kampen theorem. This is the main theorem of \cite{brown}. For this, let $b\in \indre(\Delta^1)$ and 
\begin{align*}
U_1 =& \Ob (\Cob_d^\theta)\cup (\Mor(\Cob_d^\theta)\times (\Delta^1\backslash \{b\}))\\
U_2 =& \Mor(\Cob_d^\theta)\times \indre(\Delta^1)\\
U_1\cap U_2 =& \Mor(\Cob_d^\theta)\times (\indre(\Delta^1)\backslash \{b\})\\
X_0 =& \{(W_j,\eps)\mid j\in J, \eps = 0,1\}\\
X_0' =& \{M_i\mid i\in I\}.
\end{align*}
Here $(W_j,\eps)$ should be interpreted as the point $((W_j,a),\delta)$ or $((W_j,a),1-\delta)$ in $ \Mor(\Cob_d^\theta)\times (\indre(\Delta^1)\backslash \{b\})$ for $\eps=0,1$, respectively.

According to \cite{brown}, $\pi_1(\Ob(\Cob_d^\theta)\cup (\Mor(\Cob_d^\theta)\times \Delta^1),X_0)=\pi_1(U_1\cup U_2,X_0)$ is the coequalizer in the category of groupoids of the diagram
\begin{equation*}
\pi_1(U_1\cap U_2,X_0) \rightrightarrows \pi_1(U_1,X_0)\sqcup\pi_1(U_2,X_0) \to \pi_1(U_1\cup U_2,X_0)
\end{equation*}
We begin by describing the first three groupoids in the diagram.

The map $\pi_1(U_1,X_0) \to \pi_1(U_1,X_0')$ induced by $(W_j,\eps)\mapsto \partial_\eps(W_j)$ is a vertex and piecewise surjection in the sense of \cite{higgins}, and thus it is a quotient map, according to~\cite{higgins}, Proposition 25. The kernel is the inverse image of the identity elements, i.e.\ the set 
\begin{equation*}
N=\bigsqcup_{i\in I} \{((W_{j_1},\eps_1 ), (W_{j_2},\eps_2 ))\mid j_1,j_2 \in J, \eps_1,\eps_2 \in \{0,1\}, \partial_{\eps_1}W_{j_1}=\partial_{\eps_2}W_{j_2} =M_i\}
\end{equation*}
 with multiplication 
\begin{equation*}
((W_{j_1},\eps_1 ), (W_{j_2},\eps_2 ))((W_{j_2},\eps_2 ), (W_{j_3},\eps_3 ))=((W_{j_1},\eps_1 ), (W_{j_3},\eps_3 )).
\end{equation*}
If $\pi_1(U_1,X_0)$ is replaced by $\pi_1(U_1,X_0')$ in the coequalizer diagram, $\pi_1(U_1\cup U_2,X_0)$ must be replaced by the quotient of this with the normal subgroupoid generated by $N$, c.f.\ \cite{higgins}, Proposition 27. But $N$ is also the kernel of the quotient map 
\begin{equation*}
\pi_1(U_1\cup U_2,X_0) \to \pi_1(U_1\cup U_2,X_0'),
\end{equation*}
so the new coequalizer diagram becomes
\begin{equation*}
\pi_1(U_1\cap U_2,X_0) \rightrightarrows \pi_1(U_1,X_0')\sqcup\pi_1(U_2,X_0) \to \pi_1(U_1\cup U_2,X_0').
\end{equation*}

We compute:
\begin{align*}
\pi_1(U_1\cap U_2,X_0) =& \bigsqcup_{j\in J} \pi_1(\Mor(\Cob_d^\theta),W_j)\times \{0,1\}\\
\pi_1(U_1,X_0') =& \bigsqcup_{i\in I} \pi_1(\Ob(\Cob_d^\theta),M_i)\\
\pi_1(U_2,X_0) =& \bigsqcup_{j\in J} \pi_1(\Mor(\Cob_d^\theta),W_j)\times G.
\end{align*}
Here $G=\{(i,j)\mid i,j=0,1\}$ is the groupoid with multiplication $(i,j)(j,k) = (i,k)$.

By \cite{higgins} the coequalizer, viewed as a category, is given as follows. The object set is just the set of base points $X_0'$. A morphism is represented by a sequence $x_1\dotsm x_n$ where each $x_i$ is an element of either $\pi_1(U_1,X_0')$ or $\pi_1(U_2,X_0)$ such that the target of $x_i$ coincides with the source of $x_{i+1}$ in $X_0'$. Two such sequences are equivalent if and only if they are related by a sequence of relations of the following three types:
\begin{itemize}
\item[(i)]If $e$ is an identity element in either $\pi_1(U_1,X_0')$ or $\pi_1(U_2,X_0)$, then
\begin{align*}
\dotsm x_i e x_{i+1} \dotsm  \simeq \dotsm x_i x_{i+1} \dotsm 
\end{align*}
\item[(ii)]If the product $x_ix_{i+1}=x$ makes sense in either $\pi_1(U_1,X_0')$ or $\pi_1(U_2,X_0)$, then
\begin{align*}
\dotsm x_i x_{i+1} \dotsm  \simeq \dotsm x \dotsm 
\end{align*}
\item[(iii)] Let $i_1:\pi_1(U_1\cap U_2,X_0) \to \pi_1(U_1,X_0')$ and $i_2:\pi_1(U_1\cap U_2,X_0) \to \pi_1(U_2,X_0)$ denote the inclusions. Then for $x\in \pi_1(U_1\cap U_2,X_0)$, 
\begin{align*}
\dotsm i_1(x) \dotsm \simeq \dotsm i_2(x) \dotsm
\end{align*} 
\end{itemize} 

The next step is to canonically identify such a sequence $x_1\dotsm x_n$ with a zigzag representing the same homotopy class. To each $x_i$ we associate a part of a zigzag representing $x_i$ in $\pi_1(U_1\cup U_2,X_0')$ in the following way.
If $x_i \in \pi_1(\Ob(\Cob_d^\theta ),M_l)$, let $\alpha : I \to \Ob(\Cob_d^\theta)$ be a smooth representative. This   corresponds to a morphism $W_\alpha$ with $\alpha \simeq \bar{\gamma}_{M_l\times \R} \cdot \gamma_{W_\alpha}$ by Lemma~\ref{obsti}. Otherwise $x_i$ has the form $([\gamma],(k,l))$ for some $[\gamma ] \in \pi_1(\Mor(\Cob_d^\theta ),W_j) $ and $k,l\in \{0,1\}$. We choose the following assignments:
\begin{equation*}
\begin{split}
[\alpha ] &\rightsquigarrow \cdot \xleftarrow{\alpha(0)\times \R} \cdot \xrightarrow{W_\alpha} \cdot \\
([\gamma ],(0,0 )) &\rightsquigarrow \cdot \xleftarrow{\partial_0\gamma(0) \times \R} \cdot \xrightarrow{W_{\partial_0\gamma}} \cdot \\
([\gamma ],(1,1 ))&\rightsquigarrow  \cdot  \xrightarrow{W_{\partial_1\gamma}} \cdot \xleftarrow{\partial_1\gamma(0)\times \R} \cdot \\
([\gamma ],(0,1 ))&\rightsquigarrow  \cdot  \xleftarrow{\partial_0\gamma(0)\times \R} \cdot \xrightarrow{W_{\partial_0\gamma}} \cdot \xrightarrow{W_j} \cdot \\
([\gamma ],(1,0 ))&\rightsquigarrow \cdot \xleftarrow{W_j}\cdot \xleftarrow{\partial_0\gamma(0)\times \R} \cdot \xrightarrow{W_{\partial_0\gamma}} \cdot
\end{split}
\end{equation*}
These zigzags have the correct homotopy type due to Lemma \ref{obsti}. 
Note that the manifolds $W_{\alpha}$ depend on the choice of representative $\alpha $. A different choice of representative yields a morphism that differs from $W_\alpha $ by a Type \ref{rel2} move. Hence the assignment is canonical up to Type \ref{rel2} moves. 

To each sequence $x_1 \dotsm x_n$ this associates a zigzag. We need to see that the relations (i)-(iii) on sequences correspond to performing Type \ref{rel1} and \ref{rel2} moves on the associated zigzags. This is a straightforward check, and we will only show some of the relations.\newline

 {(i)} If $e$ is the identity element in $\pi_1(\Ob(\Cob_d^\theta),M_l)$, $W_e=M_l\times \R$. Hence this relation just removes a 
\begin{equation*}
\cdot \xleftarrow{M_l\times \R} \cdot\xrightarrow{M_l\times \R} \cdot
\end{equation*}
from the zigzag. This is a Type \ref{rel1} move.\newline

{\textrm{(ii)}} If $x_i,x_{i+1}\in \pi_1(\Ob(\Cob_d^\theta),M_l)$ are represented by smooth loops $\alpha_i$ and $\alpha_{i+1}$, the relation becomes 
\begin{equation*}
\cdot \xleftarrow{M_l\times \R} \cdot \xrightarrow{W_{\alpha_i}}\cdot \xleftarrow{M_l\times \R} \cdot \xrightarrow{W_{\alpha_{i+1}}} \cdot \simeq \cdot \xleftarrow{M_l\times \R} \cdot \xrightarrow{W_{\alpha_i\cdot \alpha_{i+1}}} \cdot
\end{equation*}
 But $W_{\alpha_i\cdot \alpha_{i+1}}$ is equal to $W_{\alpha_i}\circ W_{\alpha_{i+1}}$ up to a Type \ref{rel2} move, so the zigzags differ only by  Type~\ref{rel1} and~\ref{rel2} moves.

If $x_i,x_{i+1}\in \pi_1(\Mor(\Cob_d^\theta),W_j)\times G$, there are various special cases to check. 
We shall check only
%
the case
 $x_i = ([\gamma_i],(1,0))$ and $x_{i+1} = ([\gamma_{i+1}],(0,1))$ here. Then $x_ix_{i+1}$ defines the following part of a zigzag
\begin{equation*}
\cdot \xleftarrow{W_j}\cdot \xleftarrow{\partial_0 (W_j)\times \R} \cdot \xrightarrow{W_{\partial_0\gamma_i}}\cdot \xleftarrow{\partial_0 (W_j) \times \R} \cdot \xrightarrow{W_{\partial_0\gamma_{i+1}}} \cdot \xrightarrow{W_j} \cdot
\end{equation*}
 By Type \ref{rel1} and \ref{rel2} moves, this is equivalent to
 \begin{equation*}
\cdot \xrightarrow{W_{\partial_1\gamma_i}}\cdot \xleftarrow{W_{\partial_1\gamma_{i}}} \cdot \xleftarrow{W_j}\cdot \xrightarrow{W_{\partial_0\gamma_i}}\cdot \xrightarrow{W_{\partial_0\gamma_{i+1}}} \cdot \xrightarrow{W_j} \cdot 
\end{equation*}
By Lemma \ref{diffrel}, this is again equivalent to
 \begin{equation*}
\cdot \xrightarrow{W_{\partial_1\gamma_i}}\cdot \xleftarrow{W_j} \cdot \xleftarrow{W_{\partial_0\gamma_{i}}} \cdot \xrightarrow{W_{\partial_0\gamma_i}}\cdot \xrightarrow{W_j} \cdot  \xrightarrow{W_{\partial_1\gamma_{i+1}}}\cdot
\end{equation*}
 Removing the middle part by Type \ref{rel1} moves yields
 \begin{equation*}
 \cdot \xrightarrow{W_{\partial_1(\gamma_{i} \cdot \gamma_{i+1})}} \cdot \xleftarrow{\partial_1 (W_j) \times \R} \cdot 
\end{equation*}
This corresponds to the product $([\gamma_i],(1,0))\cdot ([\gamma_{i+1}],(0,1))=([\gamma_i\cdot \gamma_{i+1}],(1,1))$.\newline

{(iii)} This is obvious from the definitions.\newline

We are now ready to prove the theorem. Let a zigzag be given. For each morphism 
\begin{equation}\label{zzW}
\dotsm  \xrightarrow{W}  \dotsm,
\end{equation}
we do as follows. First choose a smooth path $\gamma$ from $W$ to the base point $W_j$ in the $W$ component of $\Mor(\Cob_d^\theta)$. Then \eqref{zzW} is homotopic and equivalent to
\begin{equation*}
\dotsm \xleftarrow{\partial_0(W) \times \R} \cdot \xrightarrow{W_{\partial_0\gamma }} \cdot \xrightarrow{W_j}\cdot \xleftarrow{\partial_1(W)\times \R} \cdot \xrightarrow{W_{\overline{\partial_1\gamma} }}  \dotsm
\end{equation*}
by Lemma \ref{diffrel} and \ref{obsti}. 
But this zigzag is associated to a sequence $x_1\dotsm x_n$. Given another zigzag homotopic to this one, it is also equivalent to a zigzag coming from a sequence $x_1'\dotsm x_{n'}'$. We know that these sequences are related by the operations (i)--(iii), and this corresponds to doing Type \ref{rel1} and \ref{rel2} moves on the zigzags.
\end{proof}

\begin{proof}[Proof of Theorem \ref{zzrel}]
The relation \ref{rel1} certainly holds, since there is a 2-simplex in the classifying space having $\gamma_{W_1}$, $\gamma_{W_2}$ and $\gamma_{W_1\circ W_2}$ as its sides. The relation \ref{rel2} holds because the path $\gamma $ determines a homotopy between the two zigzags.

To see that these are the only relations, we apply the generalized van Kampen theorem once again. Note that the inclusion 
\begin{equation*}
\pi_1(\Ob(\Cob_d^\theta)\cup(\Mor(\Cob_d^\theta)\times \Delta^1) \cup (N_2(\Cob_d^\theta)\times \Delta^2)) \to \pi_1(B\Cob_d^\theta) 
\end{equation*}
 is an isomorphism.
 This time, let $b\in \indre(\Delta^2)$ and define
\begin{align*}
U_1=&\Ob(\Cob_d^\theta)\cup(\Mor(\Cob_d^\theta)\times \Delta^1) \cup (N_2(\Cob_d^\theta)\times \Delta^2\backslash \{b\})\\
U_2 =& N_2(\Cob_d^\theta)\times \indre(\Delta^2)\\
U_1\cap U_2 =&  N_2(\Cob_d^\theta)\times \indre(\Delta^2)\backslash \{b\}.
\end{align*} 
As base point set $X_0$, choose one representative $x_l=(W_1^l,W_2^l)$ for each element in $\pi_0(N_2(\Cob_d^\theta))$ such that $\partial_0(W_1^l)\in X_0'$ where $X_0'$ is as in the proof of Theorem \ref{zz0}. Then 
\begin{align*}
\pi_1(U_1,X_0) =&  \pi_1(\Ob(\Cob_d^\theta) \cup (\Mor(\Cob_d^\theta) \times \Delta^1),X_0)\\
\pi_1(U_2,X_0) =& \bigsqcup_{l\in L }\pi_1(N_2(\Cob_d^\theta),x_l)\\
\pi_1(U_1\cap U_2,X_0) =& \bigsqcup_{l \in L} \pi_1(N_2(\Cob_d^\theta),x_l) \times \Z.
\end{align*}
There is a map $X_0\to X_0'$ given by $(W_1^l,W_2^l)\mapsto \partial_0(W_1^l)$.
Again, this allows us to replace  $\pi_1(U_1,X_0)$ by $\pi_1(U_1,X_0') $ and $\pi_1(U_1\cup U_2,X_0)$ by $\pi_1(U_1\cup U_2,X_0')$.

Let $K= \bigsqcup_{l\in L} \{x_l\} \times \Z$ be the kernel of $i_2:\pi_1(U_1\cap U_2,X_0)\to \pi_1(U_2,X_0)$. Since $i_2$ is vertex and piecewise surjective in the sense of \cite{higgins}, Chapter 12, it is a quotient map. Thus $i_2:\pi_1(U_1\cap U_2,X_0)/K \to \pi_1(U_2,X_0)$ is an isomorphism.

Now we want to apply Proposition 27 of \cite{higgins} to compute the coequalizer of the diagram. Let $N_1(K)$ denote the normal subgroupoid of $\pi_1(U_1,X_0')$ generated by the image of $K$, and let $N_2(K)=\bigsqcup_{l\in L} \{{x_l}\}$ be the trivial normal subgroupoid of $\pi_1(U_2,X_0)$. Then there is a diagram
\begin{equation*}
K \rightrightarrows N_1(K)\sqcup N_2(K).
\end{equation*}
The coequalizer is the trivial normal subgroupoid, so by the proposition, there is a new coequalizer diagram
\begin{equation*}
\pi_1(U_1\cap U_2,X_0)/K\rightrightarrows \pi_1(U_1,X_0')/N_1(K) \sqcup \pi_1(U_2,X_0)\to \pi_1(U_1\cup U_2,X_0').
\end{equation*}
 But since $i_2:\pi_1(U_1\cap U_2,X_0)/K \to \pi_1(U_2,X_0)$ is an isomorphism, the coequalizer simply becomes $\pi_1(U_1,X_0')/N_1(K)$. This means that $\pi_1(U_1\cup U_2,X_0')$ is $\pi_1(U_1,X_0')$, which we computed in Theorem \ref{zz0}, with the only new relations being the Type~\ref{rel1} relations determined by the $x_l\in K$.
\end{proof}

\section{The chimera relations}
In this section we give another description of $\pi_1(B\Cob_d^\theta)$ in terms of generators and relations.

Let $F$ denote the free abelian group generated by diffeomorphism classes of $d$-dimensional manifolds with an equivalence class of $\theta$-structures. Let $[W]$ denote the class of $W$. 
Since $B\Cob_d^\theta$ is a loop space by \cite{GMTW}, its fundamental group is abelian.
Hence the homomorphism
\begin{equation}\label{F}
F \to \pi_1(B\Cob_d^\theta)
\end{equation}
taking $[W]$ to the homotopy class of $\gamma_W$ is well-defined by Lemma \ref{diffeo}. 

Let $W_1,W_2 \in \Cob_d^\theta(\emptyset, M)$ and $W_3,W_4 \in \Cob_d^\theta(M,\emptyset)$.
The following loops are clearly homotopic in $B\Cob_d^\theta$:
\begin{equation*}
\begin{split}
\gamma_{W_1\circ W_3} \simeq &\gamma_{W_1}\cdot \gamma_{W_3}\\
\simeq & \gamma_{W_1} \cdot \gamma_{W_4} \cdot \bar{\gamma}_{W_4} \cdot \bar{\gamma}_{W_2} \cdot \gamma_{W_2} \cdot \gamma_{W_3}\\
\simeq & \gamma_{W_1\circ W_4}\cdot \bar{\gamma }_{W_2\circ W_4} \cdot \gamma_{W_2\circ W_3}.
\end{split}
\end{equation*}

Since $\pi_1(B\Cob_d^\theta)$ is abelian, this implies:
\begin{prop}\label{chim}
For $W_1,W_2 \in \Cob_d^\theta(\emptyset, M)$ and $W_3,W_4 \in \Cob_d^\theta(M,\emptyset)$, the identity
\begin{equation}\label{chimera}
[{W_1}\circ W_3] + [W_2\circ W_4] = [W_1\circ W_4] + [W_2\circ W_3]
\end{equation}
holds in $\pi_1(B\Cob_d^\theta)$.
\end{prop}
We will refer to \eqref{chimera} as the chimera relations.\footnote{This very descriptive name is due to S\o ren Galatius.} 
Let $C$ be the subgroup of $F$ generated by the chimera relations.
Then \eqref{F} induces a homomorphism
\begin{equation}\label{F/C}
F/C \to \pi_1(B\Cob_d^\theta).
\end{equation}
We can now state the main theorem of this section.
\begin{thm}\label{pi1}
Assume that $\Cob_d^\theta$ satisfies $(ii)$ of Theorem \ref{morsti}. Then \eqref{F/C} is an isomorphism.
\end{thm}

Assuming (ii) in Theorem \ref{morsti}, \eqref{F/C} is surjective.  Indeed, the alternating zigzag 
\begin{equation*}
\cdot \xrightarrow{W_0} \cdot \xleftarrow{W_1} \dotsm \xrightarrow{W_n} \cdot
\end{equation*}
is homotopic to the image of
\begin{gather}
\begin{split}\label{mprod}
[W_0\circ \overline{W}_1\circ \dotsm \circ W_n] +& \sum_{\substack{i=1\\ i \text{ even}}}^n[W_n \circ \overline{W}_{n-1} \circ \dotsm \circ W_i \circ \overline{W}_i \circ \dotsm \circ \overline{W}_n] \\
-& \sum_{\substack{i=1\\ i \text{ odd}}}^n[W_n \circ \overline{W}_{n-1} \circ \dotsm \circ \overline{W}_i\circ {W}_i \circ \dotsm \circ \overline{W}_n]
\end{split}
\end{gather}
by the proof of Theorem \ref{morsti}. Similarly, if the zigzag starts with a morphism path in the opposite direction, just switch all signs in the sum. If $n$ is odd, the bars over the $W_n$'s should be switched. If the zigzag is not alternating, insert identity morphisms to make it alternating and apply the formula.

We want to see that the formula \eqref{mprod} defines an inverse of \eqref{F/C}.
We break the proof up in lemmas.

\begin{lem}\label{opp}
The formula \eqref{mprod} obtained from an alternating zigzag only depends on the choice of opposite morphisms $\overline{W}_i$ up to chimera relations.
\end{lem} 

\begin{proof}
Let an alternating zigzag of the form
\begin{equation*}
\cdot \xrightarrow{W_0}\cdot \xleftarrow{W_1} \dotsm  \xleftarrow{W_k} \dotsm \xleftarrow{W_n}\cdot
\end{equation*}
be given. We choose opposites $\overline{W}_i$ of $W_i$ for all $i$. Assume $\overline{W}_k'$ is a different choice of opposite to $W_k$. Define
\begin{equation*}
\begin{split}
L_0&=W_0 \circ \overline{W}_1 \circ \dotsm \circ \overline{W}_k \circ \dotsm \circ \overline{W}_n\\
N_0&=W_0 \circ \overline{W}_1 \circ \dotsm \circ \overline{W}_k' \circ \dotsm \circ \overline{W}_n,
\end{split}
\end{equation*} 
and for $1\leq i \leq n$ and $i$ odd,
\begin{align*}
L_i&=W_n \circ \overline{W}_{n-1} \circ \dotsm \circ W_k \circ \dotsm \circ W_i \circ \overline{W}_i \circ \dotsm \circ \overline{W}_k \circ \dotsm \circ \overline{W}_n\\
N_i&=W_n \circ \overline{W}_{n-1} \circ \dotsm \circ W_k \circ \dotsm \circ W_i \circ \overline{W}_i \circ \dotsm \circ \overline{W}_k' \circ \dotsm \circ \overline{W}_n.
\end{align*} 
For $i$ even, $L_i$ and $N_i$ are defined by the same  formulas except the bars over the middle $W_i$'s should be switched.

Using $\overline{W}_k$ as opposite,  \eqref{mprod} is given by
\begin{equation}\label{sti1}
\sum_{i=0}^n (-1)^i[L_i],
\end{equation}
while using $\overline{W}_k'$, it is
\begin{equation}\label{sti2}
\sum_{i=0}^n (-1)^i[N_i].
\end{equation}

Note that $N_i=L_i$ for $i>k$.
For $i\leq k$, we cut all $L_i$ between $W_{k-1}$ and $\overline{W}_k$ and all $N_i$ between $W_{k-1}$ and   $\overline{W}_k'$. Denote the parts by $L_i^{(j)}$ and $N_i^{(j)}$ for $j=1,2$ such that the $j=2$ parts contain $\overline{W}_k$ or $\overline{W}_k'$. Then for $i,l\leq k$,
\begin{align*}
L_i^{(2)}&=L_{l}^{(2)}\\
N_i^{(2)}&=N_{l}^{(2)}\\
L_i^{(1)}&=N_{i}^{(1)}.
\end{align*}
Thus there is a chimera relation
\begin{align*}
[L_i]-[L_{i+1}] =& [L_i^{(1)}\circ L_i^{(2)}] - [L_{i+1}^{(1)}\circ L_{i+1}^{(2)}] \\
\sim & [L_i^{(1)}\circ N_i^{(2)}] - [L_{i+1}^{(1)}\circ N_{i+1}^{(2)}]\\
 =& [N_i]-[N_{i+1}].
\end{align*}
Since $k$ is odd, the two sums \eqref{sti1} and \eqref{sti2} differ by $\frac{k+1}{2}$ applications of this relation.
This takes care of the case where $k$ is odd, $n$ is odd, and the first path is travelled in the positive direction. 

Changing the direction of all arrows in the zigzag only changes the signs in \eqref{sti1} and \eqref{sti2}. If $n$ is increased by one, an extra $L_{n+1} = N_{n+1}$ is added. This does not change the argument. Finally, if $k$ is even, then $L_0 = N_0$ and the remaining $L_i$ and $N_i$ are as before. There is now an even number of $1\leq i \leq k$, so the $L_i$ and $N_i$ still pair up.
\end{proof}

\begin{lem}\label{zzlim}
If two zigzags differ only by the relation \ref{rel1}, the corresponding sums \eqref{mprod} are related by chimera relations.
\end{lem}

\begin{proof}
Let a zigzag be given. After inserting identity morphisms if necessary, we assume that it is alternating of the form
\begin{equation}\label{zz}
\cdot \xrightarrow{W_0} \cdot \xleftarrow{W_1}   \dotsm  \xrightarrow{W_n} \cdot
\end{equation}
We choose opposites of all $W_i$ and define
\begin{align*}
L_0=&W_0 \circ \overline{W}_1 \circ \dotsm \circ \overline{W}_k \circ \dotsm \circ\overline{W}_n\\
L_i=&W_n \circ \overline{W}_{n-1} \circ \dotsm \circ W_k \circ \dotsm \circ W_i \circ \overline{W}_i \circ \dotsm \circ \overline{W}_k \circ \dotsm \circ \overline{W}_n.
\end{align*}
Then the zigzag \eqref{zz} corresponds to the class
\begin{equation}\label{sumL}
\sum_{i=0}^{n}{(-1)^i}[L_i].
\end{equation} 
We first consider a special case of how apply the relation \ref{rel1}. Let $k$ be odd and  $W_k\circ U = W_{k+1}$ for some $U$. After inserting identity morphisms, the original zigzag is equivalent to
\begin{equation*}
\cdot \xrightarrow{W_0} \cdot \xleftarrow{W_1}   \dotsm  \xrightarrow{W_{k-1}} \cdot \xleftarrow{1_{\partial_1(W_{k-1})}} \cdot \xrightarrow{U} \cdot \xleftarrow{W_{k+2}} \dotsm  \xrightarrow{W_n} \cdot
\end{equation*}
(If $W_{k-1}$ was an inserted identity morphism, we should really remove two identity morphisms, but this does not change \eqref{mprod}.) Let
\begin{align*}
N_0=&W_0 \circ \overline{W}_1 \circ \dotsm \circ W_{k-1}  \circ U \circ \overline{W}_{k+2} \circ \dotsm \circ \overline{W}_n\\
N_i=&W_n \circ \overline{W}_{n-1} \circ \dotsm \circ W_{k+2} \circ  \overline{U} \circ \overline{W}_{k-1} \circ W_{k-2} \circ \\
&\dotsm \circ W_i \circ \overline{W}_i \circ \dotsm \circ W_{k-1}\circ  U \circ \overline{W}_{k+2} \circ \dotsm \circ \overline{W}_n
\end{align*}
for $0 < i\leq k-1$ and 
\begin{equation*}
N_k =N_{k+1} = W_n \circ \overline{W}_{n-1} \circ \dotsm \circ W_{k+2} \circ \overline{U} \circ U \circ \overline{W}_{k+2} \circ \dotsm \circ \overline{W}_n.
\end{equation*}
For $k+2\leq i \leq n$, let $N_{i} = L_{i}$. Then the new zigzag corresponds to
\begin{equation}\label{sumN}
\sum_{i=0}^{n}{(-1)^i}[N_i].
\end{equation}
We may choose $\overline{U}= \overline{W}_{k+1}\circ W_k $. For $0\leq i\leq k-1$, we cut all $L_i$ between  $W_{k-1}$ and $\overline{W}_k$ and all $N_i$ between $W_{k-1} $ and $U$. Then there are chimera relations
\begin{equation*}
[L_i]+[N_{i+1}] \sim [L_{i+1}]+[N_i]  
\end{equation*}
for all $0\leq i\leq k-2$. There is an even number of $ i\leq k-2$. Moreover,
\begin{equation*}
 [N_{k-1}]+[L_{k}] \sim [L_{k-1}]+[L_{k+1}]
\end{equation*}
by another chimera relation. Finally, $N_{k}=N_{k+1}$ and $N_i = L_{i}$ for $i\geq k+2$. 
Hence the two sums \eqref{sumL} and \eqref{sumN} are equivalent under the chimera relations. 

If we consider the case $W_k =W_{k+1} \circ U$ instead and choose $\overline{U}= \overline{W}_{k}\circ W_{k+1} $, $[L_0]=[N_0]$ and there is a chimera relation
\begin{equation*}
[L_k]+[N_{k+1}]\sim [L_{k+1}]+[N_k].
\end{equation*}
From these cases, the statement is easily deduced for $n$ odd, $k$ even, and the case where all arrows are switched. 

We could also apply \ref{rel1} to replace 
\begin{equation*}
\cdot\xrightarrow{W_k} \cdot \xleftarrow{1_{\partial_1(W_k)}} \cdot \xrightarrow{W_{k+1}} \cdot
\end{equation*}
by 
\begin{equation*}
\cdot\xrightarrow{W_k\circ W_{k+1}} \cdot
\end{equation*}
This only removes two identical terms with opposite signs from the sum \eqref{mprod}.

All other applications of \ref{rel1} may be given as a sequence of the moves considered above.
\end{proof}

\begin{lem}\label{zzdiff}
A Type \ref{rel2} move does not change the sum \eqref{mprod}.   
\end{lem}

\begin{proof}
Let a zigzag
\begin{equation*}
\cdot \xrightarrow{W_0} \dotsm  \xrightarrow{W_{k}}  \dotsm  \xrightarrow{W_n} \cdot
\end{equation*}
be given.
 
If $W_k$ is replaced by some $W_k'$ by a Type \ref{rel2} move, we may choose $\overline{W}_k'$ equal to $\overline{W}_k$. 
Hence by Lemma \ref{diffeo}, replacing $W_k$ by $W_k'$ does not change the diffeomorphism classes in \eqref{mprod}.
\end{proof}

\begin{proof}[Proof of Theorem \ref{pi1}.]
We need to see that the surjection
\begin{equation*}
F/C \to \pi_1(B\Cob_d^\theta)
\end{equation*}
is injective. 
Consider the composition
\begin{equation*}
\pi_1(B\Cob_d^\theta) \xrightarrow{\pi} F/C \to \pi_1(B\Cob_d^\theta)
\end{equation*}
where $\pi$ is defined by the formula \eqref{mprod}. This is well-defined by Theorem \ref{zzrel} and Lemma \ref{opp}, \ref{zzlim}, and \ref{zzdiff}.
The composition is the identity so it is enough to see that ${\pi}$ is surjective.

Let $x \in F/C$. The chimera relations imply that $[W_1]+[W_2]=[W_1 \sqcup W_2]$. Thus we can represent $x$ by an element $[W]-[W']\in F$. This is $\pi(\gamma_W\cdot \bar{\gamma}_{W'})$, since we may choose $\overline{W}'=\emptyset$. 
\end{proof}

\end{document}